\documentclass[a4paper,twoside]{article}
\usepackage{a4}
\usepackage{amssymb}
\usepackage{amsmath}
\usepackage{upref} 
\usepackage[active]{srcltx}
\usepackage[pagebackref,colorlinks,citecolor=blue,linkcolor=blue]{hyperref}
\usepackage[dvipsnames]{color}
\allowdisplaybreaks[2] 
\newcount\minutes \newcount\hours
\hours=\time
\divide\hours 60
\minutes=\hours
\multiply\minutes -60
\advance\minutes \time
\newcommand{\klockan}{\the\hours:{\ifnum\minutes<10 0\fi}\the\minutes}
\newcommand{\tid}{\today\ \klockan}
\newcommand{\prtid}{\smash{\raise 10mm \hbox{\LaTeX ed \tid}}}
\renewcommand{\prtid}{}
\makeatletter
\def\sectionmark#1{} 
\def\subsectionmark#1{}
\newcommand{\sectnr}{\ifnum \c@secnumdepth >\z@
                 \thesection.\hskip 1em\relax \fi}
\def\@evenhead{\footnotesize\rm\thepage\hfil\leftmark\hfil\llap{\prtid}}
\def\@oddhead{\footnotesize\rm\rlap{\prtid}\hfil\rightmark\hfil\thepage}
\def\tableofcontents{\section*{Contents} 
 \@starttoc{toc}}
\makeatother
\makeatletter
\def\@biblabel#1{#1.}
\makeatother
\makeatletter
\let\Thebibliography=\thebibliography
\renewcommand{\thebibliography}[1]{\def\@mkboth##1##2{}\Thebibliography{#1}
\addcontentsline{toc}{section}{References}
\frenchspacing 
\setlength{\@topsep}{0pt}
\setlength{\itemsep}{0pt}%
\setlength{\parskip}{0pt plus 2pt}%
}
\makeatother
\makeatletter
\def\@seccntformat#1{\csname the#1\endcsname.\quad}
\makeatother
\newcommand{\authortitle}[2]{\author{#1}\title{#2}\markboth{#1}{#2}}
\newcommand{\auth}[2]{{#1, #2.}}
\newcommand{\art}[6]{{\sc #1, \rm #2, \it #3\/ \bf #4 \rm (#5), \mbox{#6}.}}
\newcommand{\book}[3]{{\sc #1, \it #2, \rm #3.}}
\newcommand{\AND}{{\rm and }}
\RequirePackage{amsthm}
\newtheoremstyle{descriptive}%
  {\topsep}   
  {\topsep}   
  {\rmfamily} 
  {}          
  {\bfseries} 
  {.}         
  { }         
  {}          
\newtheoremstyle{propositional}%
  {\topsep}   
  {\topsep}   
  {\itshape}  
  {}          
  {\bfseries} 
  {.}         
  { }         
  {}          
\newtheoremstyle{remarkstyle}%
  {\topsep}   
  {\topsep}   
  {\rmfamily}  
  {}          
  {\itshape} 
  {.}         
  { }         
  {}          
\theoremstyle{propositional}
\newtheorem{thm}{Theorem}
\newtheorem{cor}[thm]{Corollary}
\theoremstyle{descriptive}
\newtheorem{deff}[thm]{Definition}
\theoremstyle{remarkstyle}
\newtheorem{remark}[thm]{Remark}
\makeatletter
\renewenvironment{proof}[1][\proofname]{\par
  \pushQED{\qed}%
  \normalfont
  \trivlist
  \item[\hskip\labelsep
        \itshape
    #1\@addpunct{.}]\ignorespaces
}{%
  \popQED\endtrivlist\@endpefalse
}
\makeatother
\def\vint{\mathop{\mathchoice%
          {\setbox0\hbox{$\displaystyle\intop$}\kern 0.22\wd0%
           \vcenter{\hrule width 0.6\wd0}\kern -0.82\wd0}%
          {\setbox0\hbox{$\textstyle\intop$}\kern 0.2\wd0%
           \vcenter{\hrule width 0.6\wd0}\kern -0.8\wd0}%
          {\setbox0\hbox{$\scriptstyle\intop$}\kern 0.2\wd0%
           \vcenter{\hrule width 0.6\wd0}\kern -0.8\wd0}%
          {\setbox0\hbox{$\scriptscriptstyle\intop$}\kern 0.2\wd0%
           \vcenter{\hrule width 0.6\wd0}\kern -0.8\wd0}}%
          \mathopen{}\int}
\newcommand{\grad}{\nabla}
\DeclareMathOperator{\diam}{diam}
\DeclareMathOperator{\Div}{div}
\DeclareMathOperator{\Lip}{Lip}

\newcommand{\loc}{_{\rm loc}}
\newcommand{\Om}{\Omega}
\renewcommand{\phi}{\varphi}
\newcommand{\la}{\lambda}
\newcommand{\p}{{$p\mspace{1mu}$}}
\newcommand{\R}{\mathbf{R}}
\newcommand{\Bh}{{\widehat{B}}}
\newcommand{\Wploc}{W^{1,p}\loc}
\newenvironment{ack}{\medskip{\it Acknowledgement.}}{}

\begin{document}

\authortitle{Anders Bj\"orn and Jana Bj\"orn}
{Tensor products and sums of \p-harmonic functions, quasiminimizers
and \p-admissible weights}
\author{
Anders Bj\"orn \\
\it\small Department of Mathematics, Link\"oping University, \\
\it\small SE-581 83 Link\"oping, Sweden\/{\rm ;}
\it \small anders.bjorn@liu.se
\\
\\
Jana Bj\"orn \\
\it\small Department of Mathematics, Link\"oping University, \\
\it\small SE-581 83 Link\"oping, Sweden\/{\rm ;}
\it \small jana.bjorn@liu.se
}

\date{Preliminary version, \today}
\date{}
\maketitle

\noindent{\small {\bf Abstract}. 
The tensor product of two \p-harmonic functions is in general
not \p-harmonic, but we show that it is a quasiminimizer.
More generally, we show that the tensor product
of two quasiminimizers is a quasiminimizer.
Similar results are also obtained for quasisuperminimizers and for tensor sums.
This is done in weighted $\R^n$ with \p-admissible weights.
It is also shown that the tensor product of two
\p-admissible measures
is \p-admissible.
This last result
is 
generalized to metric spaces.
}

\bigskip
\noindent
{\small \emph{Key words and phrases}:
doubling measure,
metric space,
\p-admissible weight,
\p-harmonic function,
Poincar\'e inequality,
quasiminimizer,
quasisuperminimizer,
tensor product,
tensor sum.
}


\section{Introduction}

It is well known (and easy to prove) that the tensor product and tensor  sum
of two harmonic functions are harmonic, i.e.\
if $u_j$ is harmonic in $\Om_j \subset \R^{n_j}$, $j=1,2$,
then $u_1 \otimes u_2$ and $u_1 \oplus u_2$
are harmonic in $\Om_1 \times \Om_2 \subset \R^{n_1+n_2}$.
Here 
\[
(u_1 \otimes u_2)(x,y):=u_1(x) u_2(y)
\quad \text{and} \quad  
(u_1 \oplus u_2)(x,y):=u_1(x)+ u_2(y).
\]
It is also well known that the corresponding property for
\p-harmonic functions fails.
However, as we show in this note,
the tensor product of two \p-harmonic functions
is a quasiminimizer.

Here $u \in \Wploc(\Om)$ is \emph{\p-harmonic}
in the open set $\Om \subset \R^n$
if it is a continuous weak solution of the \p-Laplace equation
\[
     \Delta_p u:=
       \Div( |\nabla u|^{p-2} \nabla u) =0, \quad 1 < p< \infty.
\]
Moreover, $u \in \Wploc(\Om)$ is a \emph{Q-quasiminimizer} if
\[ 
      \int_{\phi \ne 0} |\grad u|^p\, dx
	\le   Q    \int_{\phi \ne 0} |\grad (u+\phi)|^p\, dx
\] 
for all boundedly supported Lipschitz functions
$\phi$ vanishing outside $\Om$.
A quasiminimizer always has a continuous representative, and
if $Q=1$ this representative is a \p-harmonic function.

In this note we show the following result.

\begin{thm}  \label{thm-tensor-prod}
Let $1<p<\infty$, and let
$u_j$ be a $Q_j$-quasiminimizer in $\Om_j\subset\R^{n_j}$
with respect to a \p-admissible weight $w_j$, $j=1,2$.
Then  $u=u_1 \otimes u_2$ and $v=u_1 \oplus u_2$
are $Q$-quasiminimizers in $\Om_1 \times \Om_2$ with respect to
the \p-admissible weight $w=w_1 \otimes w_2$, where
\begin{equation}   \label{eq-Q-intro}
    Q= \begin{cases}
      \Bigl( Q_1^{2/|p-2|} + Q_2^{2/|p-2|} \Bigr)^{|p-2|/2}, & \text{if } p \ne 2, \\
      \max\{Q_1,Q_2\}, & \text{if } p = 2.
      \end{cases}
\end{equation}
In particular, if $u_1$ and $u_2$ are \p-harmonic, then 
$u$ and $v$ are $Q$-quasiminimizers with $Q=2^{|p-2|/2}$.
\end{thm}

We also obtain a corresponding result
for quasisuperminimizers. 
We pursue our studies on weighted $\R^n$ with respect to so-called
\p-admissible weights.
To do so, we first show that
the 
product of two \p-admissible measures is \p-admissible, which
we do in  Section~\ref{sect-admissible}.
This generalizes some earlier special cases from
Lu--Wheeden~\cite[Lemma~2]{LuWhe},
Kilpel\"ainen--Koskela--Masaoka~\cite[Lemma~2.2]{KilKoMa} 
and Bj\"orn~\cite[Lemma~11]{JBFennAnn01}, but
we have not seen it proved in this form in the literature.
In fact, our result holds in the generality of metric spaces, 
see Remark~\ref{rmk-metric}.

Usually, 
$Q \ge 1$ in the definition of $Q$-quasiminimizers
but here it is convenient to also allow for $Q=0$ 
(which happens exactly when $u$ is a.e.\ constant
in every component of $\Om$).
For example, if $Q_2=0$ then $Q=Q_1$ in Theorem~\ref{thm-tensor-prod}.
Even this special case of Theorem~\ref{thm-tensor-prod}
seems to have gone unnoticed in the literature.

Quasiminimizers were introduced by
Giaquinta and Giusti~\cite{GG1}, \cite{GG2}  in the early 1980s 
 as a tool for a unified
treatment of variational integrals, elliptic equations and
quasiregular mappings on $\R^n$.
In those papers, De Giorgi's
method was extended to quasiminimizers, yielding in particular
their local H\"older continuity.
Quasiminimizers have since then been studied in a large number of papers,
first on unweighted $\R^n$ and later on metric spaces,
see Appendix~C in Bj\"orn--Bj\"orn~\cite{BBbook}
and the 
introduction in Bj\"orn~\cite{JBAdvMath}
for further discussion and references.

Quasi\-mi\-ni\-miz\-ers form
a much more flexible
class than \p-harmonic functions. 
For example, Martio--Sbordone~\cite{MaSb} showed that 
quasiminimizers
have an interesting and nontrivial theory also in one dimension,
and Kinnunen--Martio~\cite{KiMa03} developed
an interesting nonlinear potential theory for quasiminimizers, including
quasisuperharmonic functions.
Unlike \p-harmonic functions and solutions of elliptic PDEs, 
quasiminimizers can have singularities of any order, as shown in
Bj\"orn--Bj\"orn~\cite{BBpower}.

\begin{ack}
The authors 
were supported by the Swedish Research Council, grants 621-2007-6187, 621-2008-4922,
621-2014-3974 and 2016-03424.
We thank Nageswari Shanmugalingam for a fruitful discussion
concerning Theorem~\ref{thm-adm-converse}.
\end{ack}

\section{Tensor products of \texorpdfstring{\p}{p}-admissible 
measures}
\label{sect-admissible}

Let $w$ be a weight function on $\R^n$, i.e.\ a nonnegative locally 
integrable function, and let $d\mu= w\,dx$.
In this section we also let $1 \le p < \infty$ be fixed.
For a ball $B=B(x_0,r):=\{x: |x-x_0|<r\}$ in~$\R^n$ we use the 
notation $\la B=B(x_0,\la r)$.

\begin{deff}
The measure $\mu$ (or the weight $w$)
 is \emph{\p-admissible} 
if the following two conditions hold:
\begin{itemize}
\item It is \emph{doubling}, i.e.\ 
there exists a \emph{doubling constant} $C>0$ such that for all balls~$B$,
\begin{equation*}
        0 < \mu(2B) \le C \mu(B) < \infty.
\end{equation*}
\item It supports a \emph{\p-Poincar\'e inequality}, i.e.\ 
there exist constants $C>0$ and $\lambda \ge 1$
such that for all balls $B$
and all bounded locally Lipschitz functions $u$ on $\la B$,
\[ 
\vint_{B} |u-u_B| \,d\mu
    \le C \diam(B) \biggl( \vint_{\lambda B} |\nabla u|^{p} \,d\mu \biggr)^{1/p},
\] 
where  $\nabla u$ is the a.e.\ defined gradient of $u$ and
$u_B :=\vint_B u \,d\mu
:= \mu(B)^{-1}\int_B u\, d\mu$.
\end{itemize}
\end{deff}

This is one of many equivalent definitions of \p-admissible 
weights in the literature,
see e.g.\ Corollary~20.9 in Heinonen--Kilpel\"ainen--Martio~\cite{HeKiMa}
(which is not in the first edition) and
Proposition~A.17 in Bj\"orn--Bj\"orn~\cite{BBbook}.
It can be shown that 
on $\R^n$, the dilation 
$\la$ in the Poincar\'e inequality can be taken equal to 1,
see Jerison~\cite{Jerison}, Haj\l asz--Koskela~\cite{HaKo-CR} and
the discussion in \cite[Chapter~20]{HeKiMa}.

It is not known whether there exist any admissible measures on $\R^n$,
$n\ge2$, which are not absolutely continuous with respect to the
Lebesgue measure (and thus given by admissible weights).
(On $\R$ all \p-admissible measures are 
absolutely continuous, and even $A_p$ weights,
see Bj\"orn--Buckley--Keith~\cite{BjBuKe}.)
We therefore formulate our next result in terms of \p-admissible measures.

\begin{thm} \label{thm-p-adm-tensor}
Let $\mu_1$ and $\mu_2$ be \p-admissible measures on $\R^{n_1}$ and $\R^{n_2}$,
respectively. 
Then the product measure $\mu=\mu_1\times\mu_2$
is \p-admissible on $\R^{n_1+n_2}$.
\end{thm}

For a function $u$ on an open subset $\Om \subset \R^{n_1+n_2}$
we will denote the gradient by $\nabla u$.
The gradients with respect to the first $n_1$ resp.\ 
the last $n_2$ variables will be denoted by $\nabla_x u$ and $\nabla_y u$.
In this section we will only consider gradients of locally Lipschitz
functions, which are thus defined a.e.\ and coincide with the Sobolev 
gradients determined by the admissible measures, see 
Heinonen--Kilpel\"ainen--Martio~\cite[Lemma~1.11]{HeKiMa}.

\begin{proof}
Let $z=(z_1,z_2) \in \R^{n_1+n_2}$ and $r>0$.
We denote balls in $\R^{n_1}$, $\R^{n_2}$ and $\R^{n_1+n_2}$,
by $B'$, $B''$ and $B$, respectively.
Let
\[
     Q(z,r)=B'(z_1,r) \times B''(z_2,r)
\]
and note that
\begin{equation} \label{eq-Q}
     B(z,r) \subset Q(z,r) \subset B(z, \sqrt{2} r).
\end{equation}
It follows that for $B=B(z,r)$ we have
\begin{align*}
   \mu(2B) & \le \mu(Q(z,2r)) = \mu_1(B'(z_1,2r)) \mu_2(B''(z_2,2r)) \\
   & \le C \mu_1\bigl(B'\bigl(z_1,\tfrac{1}{2}r\bigr)\bigr) 
         \mu_2\bigl(B''\bigl(z_2,\tfrac{1}{2}r\bigr)\bigr) 
    = C \mu\bigl(Q\bigl(z,\tfrac{1}{2}r\bigr)\bigr)      
    \le C \mu(B),
\end{align*}
and hence $\mu$ is doubling.
Here and below, the letter $C$ denotes various positive
constants whose values may vary even within a line.

We now turn to the Poincar\'e inequality.
As mentioned above we can assume that the \p-Poincar\'e
inequalities for $\mu_1$ and $\mu_2$ hold with dilation $\la=1$.
Let $B=B(z,r)$ and  $Q=Q(z,r)= B' \times B''$.
Also let $u$ be an arbitrary bounded locally Lipschitz function on $2B$ and set
\[
c = \vint_Q u\,d\mu = \vint_{B''} \vint_{B'} u(s,t)\,d\mu_1(s)\,d\mu_2(t).
\]
Then by the Fubini theorem,
\begin{align}
\vint_Q|u-c|\,d\mu &\le \vint_{B''} \biggl( \vint_{B'} 
     \biggl| u(x,y)-\vint_{B'} u(s,y)\,d\mu_1(s) \biggr| 
           \,d\mu_1(x) \biggr) \,d\mu_2(y) \label{eq-split-c}\\
&\quad + \vint_{B''} \biggl| \vint_{B'} u(s,y)\,d\mu_1(s)
      - \vint_{B'} \vint_{B''} u(s,t)\,d\mu_2(t)\,d\mu_1(s) \biggr| \,d\mu_2(y) 
      \nonumber \\
& =: I_1 + I_2.     \nonumber
\end{align}
The first integral $I_1$ can be estimated using the \p-Poincar\'e inequality
for $\mu_1$ and $u(\,\cdot\,,y)$ on $B'$,
and then the H\"older inequality with respect to $\mu_2$, as follows
\begin{align*}
I_1 &\le \vint_{B''} Cr \biggl( \vint_{B'} |\nabla_x u(x,y)|^p\,d\mu_1(x) 
            \biggr)^{1/p} \,d\mu_2(y) \\
&\le Cr \biggl( \vint_{B''} \vint_{B'} |\nabla_x u(x,y)|^p\,d\mu_1(x) 
            \,d\mu_2(y) \biggr)^{1/p}
\le Cr \biggl( \vint_{Q} |\nabla u|^p\,d\mu \biggr)^{1/p}.
\end{align*}

As for the second integral $I_2$ in~\eqref{eq-split-c} we have 
by the Fubini theorem,
\begin{align*}
I_2 &\le \vint_{B''} \vint_{B'} \biggl| u(s,y)
      - \vint_{B''} u(s,t)\,d\mu_2(t) \biggr| \,d\mu_1(s) \,d\mu_2(y) \\
&= \vint_{B'} \vint_{B''} \biggl| u(s,y)
      - \vint_{B''} u(s,t)\,d\mu_2(t) \biggr| \,d\mu_2(y) \,d\mu_1(s),
\end{align*}
which can
be estimated in the same way as $I_1$, by switching the roles of the variables.
Thus
\[
    I_2 \le Cr \biggl( \vint_{Q} |\nabla u|^p\,d\mu \biggr)^{1/p}.
\]
Summing the estimates for $I_1$ and $I_2$ and using the doubling property
for $\mu$ we see that
\[
   \vint_B|u-c|\,d\mu \le C \vint_Q|u-c|\,d\mu
   \le Cr \biggl( \vint_{Q} |\nabla u|^p\,d\mu \biggr)^{1/p}
   \le Cr \biggl( \vint_{2B} |\nabla u|^p\,d\mu \biggr)^{1/p}.
\]
Finally, a standard argument allows us to replace $c$ on the left-hand
side by $u_{B}$ at the cost of an extra factor 2 on the right-hand side,
cf.\ \cite[Lemma~4.17]{BBbook}.
We conclude that 
$\mu$ supports a \p-Poincar\'e inequality on $\R^{n_1+n_2}$,
and thus that $w$ is \p-admissible.
\end{proof}

\begin{remark} \label{rmk-metric}
The proof of Theorem~\ref{thm-p-adm-tensor} easily generalizes to metric spaces.
More precisely, if $(X_j,d_j)$, $j=1,2$, are (not necessarily complete) 
metric  spaces equipped with
doubling measures $\mu_j$ supporting \p-Poincar\'e inequalities with 
dilation
constant $\lambda$ then $X=X_1 \times X_2$, equipped with the product 
measure $\mu=\mu_1 \times \mu_2$, supports a \p-Poincar\'e inequality
with dilation constant $2\la$ 
and $\mu$ is a doubling measure.
See e.g.\  Bj\"orn--Bj\"orn~\cite{BBbook} for the precise definitions
of these notions in metric spaces.

Poincar\'e inequalities in metric spaces are defined using
so-called upper gradients, and the main property needed for the proof
of Theorem~\ref{thm-p-adm-tensor} in the metric setting
is that whenever $g(\,\cdot\,,\cdot\,)$ is an upper gradient
of $u(\,\cdot\,,\cdot\,)$ in $X$ and $y \in X_2$, then $g(\,\cdot\,,y)$ 
is an upper gradient of $u(\,\cdot\,,y)$ with respect to $X_1$, 
and similarly for $g(x,\cdot\,)$ and $u(x,\,\cdot\,)$ with $x \in X_1$.
For this to hold, the metric on $X_1\times X_2$
can actually be defined using
\[
  d((x_1,y_1),(x_2,y_2))=\|(d_1(x_1,x_2),d(y_1,y_2))\|
\]
with an arbitrary norm $\|\cdot\|$ on $\R^2$.
In this generality we cannot assume that $\la=1$, and therefore
$\la$ also needs to be inserted at suitable places in the proof.
(If the norm does not satisfy $\|(x,0)\| \le \|(x,y)\|$ and
$\|(0,y)\| \le \|(x,y)\|$, then
the inclusions~\eqref{eq-Q} need to be modified, necessitating
similar changes also later in the proof.)
We refrain from this generalization in this note. 
Also Theorem~\ref{thm-adm-converse} below
can be similarly generalized to metric spaces.
\end{remark}

We conclude this section by showing that 
Theorem~\ref{thm-p-adm-tensor} admits a converse.

\begin{thm}   \label{thm-adm-converse}
Assume that $\mu=\mu_1 \times \mu_2$ is a \p-admissible measure
 on $\R^{n_1+n_2}$.
Then $\mu_1$ and $\mu_2$ are \p-admissible measures 
on $\R^{n_1}$ and $\R^{n_2}$, respectively. 
\end{thm}

\begin{proof}
It suffices to show the \p-admissibility of $\mu_1$.
Let $B'=(z',r)\subset\R^{n_1}$ be a ball and let 
$B'':=B(0,r)\subset\R^{n_2}$.
Let $u$ be an arbitrary bounded locally Lipschitz function on $B'$
and for $(x,y)\in B'\times B''$ define $v(x,y)=u(x)$.
Then 
\[
v_{B'\times B''} = \vint_{B'} \vint_{B''} v(x,y) \,d\mu_1(x) \,d\mu_2(y) 
    = u_{B'}.
\]
Note that for $z=(z',0)\in \R^{n_1+n_2}$,
\begin{equation}  \label{eq-comp-balls}
B(z,r) \subset B'\times B'' \subset B(z,\sqrt2 r) =: \Bh
\subset 2B'\times 2B'' \subset B(z,2\sqrt2 r).
\end{equation}
It then follows from the doubling property of $\mu$ that
\[
\mu_1(2B')\mu_2(2B'') \le \mu(B(z,2\sqrt2 r))
   \le C \mu(B(z,r)) \le C \mu_1(B')\mu_2(B'')
\]
and division by $\mu_2(2B'')\ge \mu_2(B'')$ yields
$\mu_1(2B') \le C \mu_1(B')$, i.e.\ $\mu_1$ is doubling.

As for the \-Poincar\'e inequality, we have by~\eqref{eq-comp-balls},
the doubling property of $\mu$ and \cite[Lemma~4.17]{BBbook} that
\begin{align*}
\vint_{B'} |u-u_{B'}| \,d\mu_1 &= \vint_{B'\times B''} |v-v_{B'\times B''}|\,d\mu
   \le 2 \vint_{B'\times B''} |v-v_{\Bh}|\,d\mu \\
   &\le C \vint_{\Bh} |v-v_{\Bh}|\,d\mu. 
\end{align*}
The last integral is estimated using the \p-Poincar\'e inequality
for $\mu$ and the fact that $\grad v(x,y)=\grad u(x)$ as follows
\begin{align*}
 \vint_{\Bh} |v-v_{\Bh}|\,d\mu  
  &\le C r \biggl( \vint_{\Bh} |\grad v|^p \,d\mu \biggr)^{1/p}
      \le C r \biggl( \vint_{2B'\times 2B''} |\grad v|^p \,d\mu \biggr)^{1/p} \\
  & \le C r \biggl( \vint_{2B'} |\grad u|^p \,d\mu_1 \biggr)^{1/p}.\qedhere
\end{align*}
\end{proof}

\section{Tensor products and sums of quasiminimizers}
\label{sect-tensor}

Throughout this section, $1<p<\infty$ and 
$\R^{n_j}$ is equipped with a \p-admissible weight $w_j$, $j=1,2$.
It follows from Theorem~\ref{thm-p-adm-tensor} that
$w=w_1 \otimes w_2$ is \p-admissible on $\R^{n_1+n_2}$.
We let $d\mu_j=w_j \, dx$, $j=1,2$, and $d\mu =w \,dx$.

Our aim is to prove Theorem~\ref{thm-tensor-prod}.
We will also obtain similar results for quasisuperminimizers, 
which we now define.
Let $\Om\subset\R^n$ be an open set.
By $\Lip_0(\Om)$ we denote the space of boundedly supported
Lipschitz functions vanishing outside~$\Om$.

\begin{deff}  \label{deff-quasi-super-min}
A 
function $u: \Om \to [-\infty,\infty]$
is a \emph{Q-quasi\-\/\textup{(}sub\/\textup{/}super\/\textup{)}\-minimizer}
with respect to a \p-admissible weight $w$
in a nonempty open set $\Om \subset \R^n$ if $u \in \Wploc(\Om;\mu)$ and 
\[ 
      \int_{\phi \ne 0} |\grad u|^p\, d\mu
	\le   Q    \int_{\phi \ne 0} |\grad (u+\phi)|^p\, d\mu
\] 
for all (nonpositive/nonnegative) $\phi \in \Lip_0(\Om)$.
\end{deff}

By splitting $\phi$ into its positive
and negative parts, it is easily seen that
a function is a $Q$-quasiminimizer   if and only if
it is both a $Q$-quasi\-sub\-minimizer and a $Q$-quasi\-super\-minimizer.

The Sobolev space $\Wploc(\Om;\mu)$ is defined as in 
Heinonen--Kilpel\"ainen--Martio~\cite{HeKiMa}
(although they use the letter $H$ instead of $W$). 
See \cite[Section~1.9]{HeKiMa}
and \cite[Proposition~A.17]{BBbook}
for the definition of the gradient $\grad u$ for $u \in \Wploc(\Om;\mu)$, 
which need not be the distributional gradient of $u$.

Definition~\ref{deff-quasi-super-min} is one of several
equivalent definitions of quasi(sub/super)minimizers,
see Bj\"orn~\cite[Proposition~3.2]{ABkellogg}, where this was shown
on metric spaces.
It follows from Propositions~A.11 and~A.17 in \cite{BBbook} that the 
metric space definitions coincide with the usual ones on weighted $\R^n$
(with a \p-admissible weight).

For quasisuperminimizers, an analogue of Theorem~\ref{thm-tensor-prod} 
takes the following form.

\begin{thm}  \label{thm-tensor-prod-super}
Let $u_j$ be a $Q_j$-quasisuperminimizer in $\Om_j\subset\R^{n_j}$
with respect to \p-admissible weights  $w_j$,
$j=1,2$, and $Q$ be given by \eqref{eq-Q-intro}.
Then $u_1 \oplus u_2$ is a $Q$-quasisuperminimizer in 
$\Om=\Om_1 \times \Om_2$ with respect to $w= w_1 \otimes w_2$.

In addition, if both $u_1$ and $u_2$ are nonnegative/nonpositive,
then $u_1 \otimes u_2$ is a 
$Q$-quasisuper/sub\-mi\-ni\-mizer  in~$\Om$ with respect to $w$.
\end{thm}

By considering $-u_1$ and $-u_2$, we easily obtain a corresponding
result for quasisubminimizers.
Usually, 
$Q_j \ge 1$ but we also allow for $Q_j=0$.
This can only happen when $u_j$ is constant (a.e.\ in each
component of $\Om_j$),
but when this is fulfilled 
in Theorem~\ref{thm-tensor-prod} or~\ref{thm-tensor-prod-super}
it immediately implies the following conclusion.

\begin{cor} \label{cor-u-const}
  If $u$ is a $Q$-quasi\/\textup{(}super\/\textup{)}minimizer
in $\Om\subset\R^{n_1}$
with respect to a \p-admissible weight $w_1$,
and we let 
\(
v(x,y)=u(x)
\)
for $(x,y)\in\Om\times\R^{n_2}$, 
then $v$ is a $Q$-quasi\/\textup{(}super\/\textup{)}minimizer 
in $\Om\times\R^{n_2}$
with respect to $w=w_1 \otimes w_2$,
whenever $w_2$ is a \p-admissible weight on $\R^{n_2}$.
\end{cor}

\begin{proof}
As $v =  u \oplus \mathbf{0}$,
where $\mathbf{0}$ is the zero function,
this follows directly from 
Theorems~\ref{thm-tensor-prod} and~\ref{thm-tensor-prod-super}.
\end{proof}

\begin{proof}[Proof of Theorem~\ref{thm-tensor-prod}]
Since $u_1$ and $u_2$ are finite a.e., and the quasiminimizing
property is the same for all representatives of an equivalence class
in the local 
Sobolev space, we may assume that $u_1$ and $u_2$ are finite
everywhere.

First, we show that $u:=u_1 \otimes u_2$ is a $Q$-quasiminimizer.
Note that  
\[
|\grad u(x,y)|^p = (|\grad_x u(x,y)|^2 + |\grad_y u(x,y)|^2)^{p/2},
\]
where $\grad_x u(x,y) = u_2(y) \grad u_1(x)$ and 
$\grad_y u(x,y) = u_1(x) \grad u_2(y)$.

Let $\phi\in\Lip_0(\Om)$ be arbitrary. 
For a fixed $y\in\Om_2$, let 
\[
\Om_1^y=\{x\in\Om_1: \phi(x,y)\ne0\}.
\]
As $u_1$ is a $Q_1$-quasiminimizer in $\Om_1$, so is 
$u(\,\cdot\,,y)=u_2(y)u_1(\,\cdot\,)$.
Since $\phi(\,\cdot\,, y)\in \Lip_0(\Om_1^y)$, we get
\[
\int_{\Om_1^y} |\nabla_x u(x,y)|^p \,d\mu_1(x)
       \le Q_1 \int_{\Om_1^y} |\grad_x (u(x,y)+\phi(x,y))|^p\,d\mu_1(x).
\]
Integrating over all $y\in\Om_2$ with nonempty $\Om_1^y$ yields 
\begin{equation}
\int_{\phi\ne0} |\grad_x u|^p\,d\mu
       \le Q_1 \int_{\phi\ne0} |\grad_x (u+\phi)|^p\,d\mu.
\label{est-grad-u1}
\end{equation}
Similarly,
\begin{equation}
\int_{\phi\ne0} |\grad_y u|^p\,d\mu
       \le Q_2 \int_{\phi\ne0} |\grad_y (u+\phi)|^p\,d\mu.
\label{est-grad-u2}
\end{equation}
Now we consider four cases.

\medskip
\emph{Case} 1. $Q_1=0$.
In this case, $\nabla u_1 \equiv 0$ a.e., and so $\nabla_x u \equiv 0$
a.e.
Hence, by \eqref{est-grad-u2},
\[
\int_{\phi\ne0} |\nabla u|^p\,d\mu
   = \int_{\phi\ne0} |\grad_y u|^p\,d\mu
   \le Q_2 \int_{\phi\ne0} |\grad_y (u+\phi)|^p\,d\mu
   \le Q_2 \int_{\phi\ne0} |\nabla (u+\phi)|^p\,d\mu,
\]
and thus $u$ is a $Q_2$-quasiminimizer.

\medskip
\emph{Case} 2. $Q_2=0$.
This is similar to Case~1.

\medskip
\emph{Case} 3. $p \le 2$.
In this case, summing \eqref{est-grad-u1} and
\eqref{est-grad-u2} gives 
\begin{align*}
\int_{\phi\ne0} |\grad u|^p \,d\mu
  &\le \int_{\phi\ne0} (|\grad_x u|^p + |\grad_y u|^p) \,d\mu \\
       &\le  \int_{\phi\ne0} (Q_1 |\grad_x (u+\phi)|^p
            + Q_2 |\grad_y (u+\phi)|^p ) \,d\mu.
\end{align*}
This proves the result for $p=2$.  For $p<2$, the H\"older inequality
applied to the sum $Q_1a^p + Q_2b^p$ in the last integrand shows that
\begin{align*}
\int_{\phi\ne0} |\grad u|^p \,d\mu
   &
\le \Bigl( Q_1^{2/(2-p)} + Q_2^{2/(2-p)} \Bigr)^{1-p/2} \\
& \quad \times                 \int_{\phi\ne0} (|\grad_x (u+\phi)|^2
            + |\grad_y (u+\phi)|^2)^{p/2} \,d\mu \\
   &
   = \Bigl( Q_1^{2/(2-p)} + Q_2^{2/(2-p)} \Bigr)^{1-p/2} 
                 \int_{\phi\ne0} |\grad (u+\phi)|^p \,d\mu.
\end{align*}

\medskip
\emph{Case} 4. $p \ge 2$ and $Q_1,Q_2 >0$.
Rewrite $|\grad u|^p$ as 
\[
|\grad u|^p = (|\grad_x u|^2 + |\grad_y u|^2)^{p/2}
   = \biggl( Q_1^{2/p} \biggl(\frac{1}{Q_1}\biggr)^{2/p}|\grad_x u|^2 
   + Q_2^{2/p} \biggl(\frac{1}{Q_2}\biggr)^{2/p}|\grad_y u|^2 \biggr)^{p/2}.
\]
The H\"older inequality applied to the sum $Q_1^{2/p}a^2 + Q_2^{2/p}b^2$ 
implies
\[
|\grad u|^p 
        \le \Bigl( Q_1^{2/(p-2)} + Q_2^{2/(p-2)}\Bigr)^{(p-2)/2}
   \biggl(\frac{1}{Q_1}|\grad_x u|^p + \frac{1}{Q_2}|\grad_y u|^p\biggr).
\]
Integrating over the set $\{(x,y)\in\Om:\phi(x,y)\ne0\}$ 
and using \eqref{est-grad-u1} and \eqref{est-grad-u2} we  obtain
\begin{align*}
\int_{\phi\ne0} |\grad u|^p \,d\mu 
&
    \le \Bigl( Q_1^{2/(p-2)} + Q_2^{2/(p-2)}\Bigr)^{(p-2)/2} \\
&     \quad \times
          \int_{\phi\ne0} (|\grad_x (u+\phi)|^p
            + |\grad_y (u+\phi)|^p ) \,d\mu.
\end{align*}
As $p/2 \ge 1$, the elementary inequality $a^p+b^p\le (a^2+b^2)^{p/2}$
concludes the proof for $u$.

\medskip

We now turn to $v:=u_1 \oplus u_2$.
Let $\phi\in\Lip_0(\Om)$ be arbitrary. 
Note that   
\[   
|\grad v(x,y)|^p
= (|\grad_x v(x,y)|^2 + |\grad_y v(x,y)|^2)^{p/2}
= (|\grad u_1(x)|^2 + |\grad u_2(y)|^2)^{p/2} 
\]   
and 
\[ 
|\grad (v+\phi)|^p = (|\grad_x (v+\phi)|^2 + |\grad_y (v+\phi)|^2)^{p/2}.
\] 
For a fixed $y\in\Om_2$, let 
\[
\Om_1^y=\{x\in\Om_1: \phi(x,y)\ne0\}.
\]
As $u_1$ is a $Q_1$-quasiminimizer in $\Om_1$ and
$\phi(\,\cdot\,, y)\in \Lip_0(\Om_1^y)$, we get
\[
\int_{\Om_1^y} |\grad u_1(x)|^p\,d\mu_1(x) 
       \le Q_1 \int_{\Om_1^y} |\grad_x (u_1(x,y)+\phi(x,y))|^p\,d\mu_1(x).
\]
Integrating over all $y\in\Om_2$ with nonempty $\Om_1^y$ yields 
\begin{equation*}
\int_{\phi\ne0} |\grad u_1|^p\,d\mu_1(x) \,d\mu_2(y)
       \le Q_1 \int_{\phi\ne0} |\grad_x (v+\phi)|^p\,d\mu_1(x)\,d\mu_2(y),
\end{equation*}
i.e.\ \eqref{est-grad-u1} holds.
Similarly, \eqref{est-grad-u2} holds and the rest of the proof is as
for $u$.
\end{proof}

\begin{proof}[Proof of Theorem~\ref{thm-tensor-prod-super}]
This proof is very similar to the proof above.
In this case we of course assume that $\phi \in \Lip_0(\Om)$
is nonnegative/nonpositive.

The only other difference in the proof is that since 
$u_1$ is a $Q_1$-quasi\-super\-minimizer in $\Om_1$ 
and $u_2(y)$ is nonnegative/nonpositive,
we can conclude that 
\[
u(\,\cdot\,,y)=u_2(y)u_1(\,\cdot\,)
\]
is a $Q_1$-quasisuper/subminimizer in $\Om_1$.
The rest of the proof is the same;
in particular the proof for $v$ needs no nontrivial changes, and is thus
valid also when $u_1$ and $u_2$ change sign.
\end{proof}

For tensor sums one can use Theorem~\ref{thm-tensor-prod-super}
to deduce (the corresponding part of) Theorem~\ref{thm-tensor-prod}.
For tensor products this is not possible as in this case the
quasisuperminimizers in Theorem~\ref{thm-tensor-prod-super}
need to be nonnegative.
This nonnegativity is an essential assumption for quasisuperminimizers,
which is not required for quasiminimizers.
(To see this consider what happens when $u_2 \equiv -1$.)
We can however obtain the following result.

\begin{thm}
Let $u_1$ be a 
$Q_1$-quasisub/super\-minimizer 
in $\Om_1$ and $u_2 \ge 0$ be a $Q_2$-quasiminimizer in $\Om_2$,
with respect to \p-admissible weights $w_1$ and $w_2$, respectively.

Then $u_1 \otimes u_2$ is a 
$Q$-quasisub/super\-minimizer in $\Om= \Om_1 \times \Om_2$
with respect to $w=w_1 \otimes w_2$, 
where $Q$ is given by \eqref{eq-Q-intro}.
\end{thm}

\begin{proof}
This is proved using a similar modification of the proof
of Theorem~\ref{thm-tensor-prod} as we
did when proving  Theorem~\ref{thm-tensor-prod-super}.
The key fact is that quasiminimizers are preserved 
under multiplication by real numbers, while
the corresponding fact for quasi\-sub/super\-minimizers is only 
true under multiplication by nonnegative real numbers.
\end{proof}

\end{document}